\documentclass[oneside,english,9pt]{amsart}
\usepackage[T1]{fontenc}
\usepackage[latin9]{inputenc}
\usepackage{geometry}
\usepackage{amstext}
\usepackage{amsthm}
\usepackage{amssymb}

\makeatletter
\numberwithin{equation}{section}
\numberwithin{figure}{section}
\theoremstyle{plain}
\newtheorem{thm}{\protect\theoremname}
  \theoremstyle{definition}
  \newtheorem{defn}[thm]{\protect\definitionname}
  \theoremstyle{remark}
  \newtheorem{rem}[thm]{\protect\remarkname}
  \theoremstyle{plain}
  \newtheorem{prop}[thm]{\protect\propositionname}
  \theoremstyle{plain}
  \newtheorem{lem}[thm]{\protect\lemmaname}
  \theoremstyle{plain}
  \newtheorem{cor}[thm]{\protect\corollaryname}
  \theoremstyle{remark}
  \newtheorem*{claim*}{\protect\claimname}

\usepackage{xypic}

\usepackage{babel}
\providecommand{\claimname}{Claim}
  \providecommand{\corollaryname}{Corollary}
  \providecommand{\definitionname}{Definition}
  \providecommand{\lemmaname}{Lemma}
  \providecommand{\propositionname}{Proposition}
  \providecommand{\remarkname}{Remark}
\providecommand{\theoremname}{Theorem}

\title{Abstract key polynomials and comparison theorems with the key polynomials
of Mac Lane -- Vaqui\'e}

\g@addto@macro{\endabstract}{\@setabstract}
\newcommand{\authorfootnotes}{\renewcommand\thefootnote{\@fnsymbol\c@footnote}}%

\makeatother

\begin{document}
\maketitle

\begin{center}
\par\end{center}

\begin{center}
\authorfootnotes {\LARGE{}J. Decaup\textsuperscript{{\LARGE{}1}},
W. Mahboub\textsuperscript{{\LARGE{}2}}, M. Spivakovsky\textsuperscript{{\LARGE{}1}} }
\par\end{center}{\LARGE \par}

\begin{center}
{\LARGE{}\bigskip{}
}
\par\end{center}{\LARGE \par}

\begin{center}
{\LARGE{}\textsuperscript{{\LARGE{}1}} Institut de Math\'ematiques
de Toulouse/CNRS UMR 5219, Universit\'e Paul Sabatier, 118, route de
Narbonne, 31062 Toulouse cedex 9, France. }
\par\end{center}{\LARGE \par}

\begin{center}
{\LARGE{}\textsuperscript{{\LARGE{}2}}Department of Mathematics,
American University of Beirut.}
\par\end{center}{\LARGE \par}

\begin{center}
{\LARGE{}\bigskip{}
}
\par\end{center}{\LARGE \par}

{\LARGE{}

\section*{Introduction}

{\LARGE{}Let $K$ be a valued field and $K\hookrightarrow K(x)$ a
simple purely transcendental extension of $K$. In the nineteen thirties,
S. Mac Lane considered the special case when the valuation $\nu$
of $K$ is disrete of rank one and defined the notion of }\textbf{\LARGE{}key
polynomials}{\LARGE{}, associated to various extensions of $\nu$
to $K(x)$ (\cite{ML1} and \cite{ML}). Key polynomials are elements
of $K[x]$ which describe the structure of various extensions of $\mu$
to $K(x)$ and the relationship between them. Roughly speaking, they
measure how far a given extension of $\mu$ to $K(x)$ is from the
monomial valuation (the one that assigns to each polynomial $f\in K[x]$
the minimal value of the monomials appearing in $f$). Mac Lane's definition
of key polynomials was axiomatic: an element $f\in K[x]$ is a key
polynomial for an extension $\mu$ of $\nu$ to $K$ if it is monic,
$\mu$-minimal and $\mu$-irreducible (see Section \ref{KeyP} below
for precise definitions).}{\LARGE \par}

{\LARGE{}Michel Vaqui\'e (\cite{V0}, \cite{V}, \cite{V1} and \cite{V2})
extended this definition to the case of arbitrary valued fields $K$
(that is, without the assumption that $\nu$ is discrete). One important
difference with the case of discrete valuations treated by Mac Lane
is the presence of limit key polynomials (which will not be discussed
in the present paper).}{\LARGE \par}

{\LARGE{}F. H. Herrera, M. A. Olalla, W. Mahboub and M. Spivakovsky
defined a different, though closely related notion of key polynomials
(\cite{HOS} and \cite{HMOS}). In their approach the emphasis was
on describing key polynomials by explicit formulae and on constructing
the successive key polynomials recursively in terms of the preceding
ones.}{\LARGE \par}

{\LARGE{}In his Ph.D. thesis (Toulouse 2013), W. Mahboub proved comparison
theorems between Mac Lane -- Vaqui\'e key polynomials.}{\LARGE \par}

{\LARGE{}Apart from a better understanding of the structure of simple
extensions of valued fields in its own right, one of the intended
applications of the theory of key polynomials is the work towards
the proof of the Local Uniformization Theorem over fields of arbitrary
characteristic. Jean-Christophe San Saturnino (see Theorem 6.5 of
\cite{JCSS}) proved that in order to achieve Local Uniformization
of a variety embedded in $Spec\ k[u_{1},\dots,u_{n}]$ along a given
valuation $\mu$ of $k(u_{1},\dots,u_{n})$ it is sufficient to monomialize
the first limit key polynomial of the simple extension $k(u_{1},\dots,u_{n-1})\hookrightarrow k(u_{1},\dots,u_{n})$
(assuming local uniformization is already known in ambient dimension
at most $n-1$). Although limit key polynomials are beyond the scope
of this paper, we hope that the comparison theorems proved here will
clarify the relationship between different definitions of key polynomials
and therefore be useful for applications.}{\LARGE \par}

{\LARGE{}Let $\mu$ be an extension of $\nu$ to $K$. In this paper
we give a new definition of key polynomials (which we call }\textbf{\LARGE{}abstract
key polynomials}{\LARGE{}) associated to $\mu$ and study the relationship
between them and key polynomials of Mac Lane -- Vaqui\'e. Associated
to each abstract key polynomial $Q$, we define the }\textbf{\LARGE{}truncation}{\LARGE{}
$\mu_{Q}$ of $\mu$ with respect to $Q$. Roughly speaking, $\mu_{Q}$
is an approximation to $\mu$ defined by $Q$. This approximation
gets better as $\deg_{x}Q$ and $\mu(Q)$ increase. We also define
the notion of an abstract key polynomial $Q'$ being an }\textbf{\LARGE{}immediate
successor}{\LARGE{} of another abstract key polynomial $Q$ (in this
situation we write $Q<Q'$). The main comparison results proved in
this paper are as follows:}{\LARGE \par}

{\LARGE{}Theorem \ref{abstractimpliesVaquie}: An abstract key polynomial
for $\mu$ is a Mac Lane -- Vaqui\'e key polynomial for the truncated
valuation $\mu_{Q}$.}{\LARGE \par}

{\LARGE{}Theorem \ref{successorimpliesVaquie}: If $Q<Q'$ are two
abstract key polynomials for $\mu$ then $Q'$ is a Mac Lane -- Vaqui\'e
key polynomial for $\mu_{Q}$.}{\LARGE \par}

{\LARGE{}Theorem \ref{vaquieimpliesabstract} which, for a monic polynomial
$Q\in K[x]$ and a valuation $\mu'$ of $K(x)$, gives a sufficient
condition for $Q$ to be an abstract key polynomial for $\mu'$. Combined
with $\mathrm{Proposition\ 1.3}$ of \cite{V}, this describes a class
of pairs of valuations $(\mu,\mu')$ such that $Q$ is a Mac Lane
-- Vaqui\'e key polynomial for $\mu$ and an abstract key polynomial
for $\mu'$. This can be regarded as a partial converse to Theorem
\ref{successorimpliesVaquie}.}{\LARGE \par}

{\LARGE{}This paper is structured as follows. In \S\ref{KeyP} we
define the Mac Lane -- Vaqui\'e and the abstract key polynomials and
study their properties. In \S\ref{relationship} we prove our main
comparison results stated above.}{\LARGE \par}

\section{Preliminaries and notation}

{\LARGE{}Throughout this paper, $\mathbb{N}$ will denote the non-negative
integers, $\mathbb{N}^{*}$ the strictly positive integers. For a
field $L$, the notation $L^{*}$ will stand for the multiplicative
group $L\setminus\{0\}$. }{\LARGE \par}
\begin{itemize}
\item {\LARGE{}Let $R$ be a domain, $K$ the field of fractions of $R$,
$\mu$ a valuation of $K$ with value group $\Gamma$ and $\alpha\in\Gamma$.
We define: }{\LARGE \par}\end{itemize}
\begin{enumerate}
\item {\LARGE{}$\mathcal{P_{\alpha}}(R):=\left\{ x\in R\text{ such that }\mu(x)\geq\alpha\right\} $ }{\LARGE \par}
\item {\LARGE{}$\mathcal{P}_{\alpha^{+}}(R):=\left\{ x\in R\text{ such that }\mu(x)>\alpha\right\} $ }{\LARGE \par}
\item {\LARGE{}$\mathrm{gr_{\mu}(R):=\bigoplus\limits _{\alpha\in\Gamma}\frac{\mathcal{P}_{\alpha}(R)}{\mathcal{P}_{\alpha^{+}}(R)}}$ }{\LARGE \par}
\item {\LARGE{}$G_{\mu}:=\mathrm{gr}_{\mu}(K)$ }{\LARGE \par}
\item {\LARGE{}For each $f\in R$ such that $\mu(f)=\alpha$, we denote
by $\mathrm{in}_{\mu}(f)$ the image of $f$ in $\frac{\mathcal{P}_{\alpha}(R)}{\mathcal{P}_{\alpha^{+}}(R)}$;
we call this image the }\textbf{\LARGE{}initial form}{\LARGE{} of
$f$ with respect to $R$ and $\mu$. }{\LARGE \par}\end{enumerate}
\begin{itemize}
\item {\LARGE{}Let $K\hookrightarrow K(x)$ be a purely transcendental extension
of $K$. Let $Q$ be a monic polynomial in $K[x]$. Every polynomial
$g\in K[x]$ can be written in a unique way as 
\begin{equation}
g=\sum\limits _{j=0}^{s}g_{j}Q^{j},\label{eq:Qexpansion}
\end{equation}
with all the $g_{j}\in K[x]$ of degree strictly less than $\deg(Q)$.
We call (\ref{eq:Qexpansion}) the $Q$}\textbf{\LARGE{}-expansion}{\LARGE{}
of $g$. }{\LARGE \par}\end{itemize}
\begin{defn}
{\LARGE{}Let $g=\sum\limits _{j=0}^{s}g_{j}Q^{j}$ be the $Q$-expansion
of an element $g\in K[x]$. We put $\mu_{Q}(g):=\min\limits _{\underset{g_{j}\neq0}{0\leq j\leq s}}\mu(g_{j}Q^{j})$
and we call $\mu_{Q}$ the }\textbf{\LARGE{}truncation}{\LARGE{} of
$\mu$ with respect to $Q$. }{\LARGE \par}\end{defn}
\begin{itemize}
\item {\LARGE{}Let $\mu$ be a valuation of the field $K(x)$, where $x$
is an algebraically independent element over a field $K$. Consider
the restriction of $\mu$ to $K[x]$. Consider a monic polynomial
$Q\in K[x]$. Assume that $\mu_{Q}$ is a valuation (below we will
define the notion of abstract key polynomial and will show that $\mu_{Q}$
is always a valuation in that case). Fix another polynomial $f$ and
let $f=\sum\limits _{j=0}^{s}f_{j}Q^{j}\in K[x]$ be the $Q$-expansion
of $f$. Let $\alpha=\mu_{Q}(f)$. We denote by $\mathrm{In}_{Q}f$
the element $\sum\limits _{\mu(f_{j}Q^{j})=\alpha}f_{j}Q^{j}\in K[x]$.
Note that, by definition, $\mathrm{In}_{Q}f\in K[x]$, while $\mathrm{in}_{\mu_{Q}}f\in\mathrm{gr}_{\mu_{Q}}K[x]$.
We have $\mathrm{in}_{\mu_{Q}}\left(\mathrm{In}_{Q}f\right)=\mathrm{in}_{\mu_{Q}}f$. }{\LARGE \par}
\end{itemize}

\section{Key Polynomials}

{\LARGE{}\label{KeyP}}{\LARGE \par}

\subsection{Key polynomials of Mac Lane--Vaqui\'e\label{VML}}

{\LARGE{}We first recall the notion of key polynomial, introduced
by Vaqui\'e in \cite{V}, generalizing an earlier construction of Mac
Lane \cite{ML}. }{\LARGE \par}
\begin{defn}
{\LARGE{}Let $\left(f,g\right)\in K[x]^{2}$. We say that $f$ and
$g$ are $\mu$-equivalent and we write $f\sim_{\mu}g$ if $f$ and
$g$ have the same initial form with respect to $K$ and $\mu$. }{\LARGE \par}\end{defn}
\begin{rem}
{\LARGE{}The polynomials $f$ and $g$ are $\mu$-equivalent if and
only if \label{equivin}
$$
\mu(f-g)>\mu(f)=\mu(g).
$$
Indeed, if 
\begin{equation}
\mathrm{in}_{\mu}f=\mathrm{in}_{\mu}g\label{eq:initialequal}
\end{equation}
then, in particular, $\mu(f)=\mu(g)$. Furthermore, (\ref{eq:initialequal})
says that $f$ and $g$ agree modulo $\mathcal{P}_{\mu(f)^{+}}$.
Thus 
\[
\mu(f-g)>\mu(f)=\mu(g).
\]
Conversely, if $\mu(f-g)>\mu(g)$, then $\mathrm{in}_{\mu}(f-g+g)=\mathrm{in}_{\mu}(g)$. }{\LARGE \par}\end{rem}
\begin{defn}
{\LARGE{}Let $\left(f,g\right)\in K[x]^{2}$. We say that $g$ is
$\mu$-divisible by $f$ or that $f$ $\mu$-divides $g$ (denoted
by $f\ |_{\mu}\ g$) if the initial form of $g$ with respect to $\mu$
is divisible by the initial form of $f$ with respect to $\mu$ in
$\mathrm{gr}_{\mu}K[x]$. }{\LARGE \par}\end{defn}
\begin{rem}
{\LARGE{}We have $f\ |_{\mu}\ g$ if and only if there exists $c\in K[x]$
such that $g\sim_{\mu}fc$. }{\LARGE \par}\end{rem}
\begin{defn}
{\LARGE{}Let $Q\in K[x]$ be a monic polynomial. We say that $Q$
is a Mac Lane--Vaqui\'e key polynomial for the valuation $\mu$ if the
following conditions hold: }{\LARGE \par}\end{defn}
\begin{enumerate}
\item {\LARGE{}$Q$ is $\mu$-irreducible, that is, for any $g,h\in K[x]$,
if $Q\ |_{\mu}\ gh$, then $Q\ |_{\mu}\ g$ or $Q\ |_{\mu}\ h$. }{\LARGE \par}
\item {\LARGE{}$Q$ is $\mu$-minimal, that is, for every $f\in K[x]$,
if $Q\ |_{\mu}\ f$ then $\deg(f)\geq\deg(Q)$. }{\LARGE \par}\end{enumerate}
\begin{prop}
{\LARGE{}Let $P$ be an element of $K[x]$. Assume that $P$ is $\mu$-irreducible.
Then $\mathrm{in}_{\mu}P$ is irreducible in $\mathrm{gr}_{\mu}K[x]$. }{\LARGE \par}\end{prop}
\begin{proof}
{\LARGE{}Assume that $\mathrm{in}_{\mu}P$ is reducible in $\mathrm{gr}_{\mu}K[x]$,
aiming for contradiction. Write $\mathrm{in}_{\mu}P=\mathrm{in}_{\mu}g\text{ }\mathrm{in}_{\mu}h$
with $\mu(g),\mu(h)>0$.}{\LARGE \par}

{\LARGE{}We have $P\ |_{\mu}\ gh$, but $P\ \nmid_{\mu}\ g$ and $P\ \nmid_{\mu}\ h$.
This contradicts the $\mu$-irreducibility of $P$. The Proposition
is proved. }{\LARGE \par}\end{proof}
\begin{rem}
{\LARGE{}\label{UFD} Assume that every homogeneous element of $\mathrm{gr}_{\mu}K[x]$
admits a unique decomposition into irreducible factors. Then $Q$
is $\mu$-irreducible if and only if its initial form with respect
to $\mu$ is irreducible. }{\LARGE \par}
\end{rem}
{\LARGE{}
We now introduce an alternative, though closely related notion of key polynomials.}{\LARGE \par}

\subsection{Abstract key polynomials }

{\LARGE{}We keep the same notation as in \ref{VML}, and we add the
following: }{\LARGE \par}
\begin{enumerate}
\item {\LARGE{}For each strictly positive integer $b$, we write $\partial_{b}:=\frac{\partial^{b}}{b!\partial x^{b}}$,
the so-called $b$-th }\textbf{\LARGE{}formal derivative}{\LARGE{}
with respect to $x$. }{\LARGE \par}
\item {\LARGE{}For each polynomial $P\in K[x]$, let $\epsilon_{\mu}(P):=\max\limits _{b\in\mathbb{N}^{\ast}}\left\{ \frac{\mu(P)-\mu(\partial_{b}P)}{b}\right\} $ }{\LARGE \par}
\item {\LARGE{}For each polynomial $P\in K[x]$, let $b(P):=\min I(P)$
where 
\[
I(P):=\left\{ b\in\mathbb{N^{\ast}}\text{ such that }\frac{\mu(P)-\mu(\partial_{b}P)}{b}=\epsilon_{\mu}(P)\right\} .
\]
}{\LARGE \par}\end{enumerate}
\begin{defn}
{\LARGE{}Let $Q$ be a monic polynomial in $K[x]$. We say that $Q$
is an }\textbf{\LARGE{}abstract key polynomial}{\LARGE{} for $\mu$
if for each polynomial $f$ satisfying
$$
\epsilon_{\mu}(f)\geq\epsilon_{\mu}(Q),
$$
we have $\deg(f)\geq\deg(Q)$. }{\LARGE \par}\end{defn}
\begin{prop}
{\LARGE{}\label{rem:prod} Let $t\geq2$ be an integer, $Q$ an abstract
key polynomial and $P_{1},\ldots,P_{t}\in K[x]$ of degrees strictly
less than $\deg(Q)$. Let $\prod\limits _{i=1}^{t}P_{i}=qQ+r$ be
the Euclidean division of $\prod\limits _{i=1}^{t}P_{i}$ by $Q$.
Then 
\[
\mu(r)=\mu\left(\prod\limits _{i=1}^{t}P_{i}\right)<\mu(qQ).
\]
}{\LARGE \par}\end{prop}
\begin{proof}
{\LARGE{}We proceed by induction on $t$.}{\LARGE \par}

{\LARGE{}First, consider the case $t=2$. We want to show that
$$
\mu(P_{1}P_{2})=\mu(r)<\mu(qQ).
$$
Assume the contrary, that is, $\mu(P_{1}P_{2})\geq\mu(qQ)$ and $\mu(r)\geq\mu(qQ)$.
For each $j\in\mathbb{N}^{*}$, we have $\mu(\partial_{j}P_{1})>\mu(P_{1})-j\epsilon_{\mu}(Q)$,
and similarly for $P_{2},q,r$, because all these polynomials have
degree strictly less than $\deg(Q)$ and $Q$ is an abstract key polynomial.
Since $\mu(\partial_{j}q)>\mu(q)-j\epsilon_{\mu}(Q)$ for all strictly
positive integers $j$, we deduce that 
\[
\mu(q\partial_{b(Q)}Q)=\mu(q)+\mu(Q)-b(Q)\epsilon_{\mu}(Q)<\mu(\partial_{b(Q)-j}Q)+\mu(\partial_{j}q)
\]
for all $j\in\{1,\dots,b(Q)\}$.\\
Hence $\mu\left(\partial_{b(Q)}(qQ)\right)=\mu\left(\sum\limits _{j=0}^{b(Q)}\left(\partial_{b(Q)-j}q\partial_{j}Q\right)\right)=\mu(q\partial_{b(Q)}Q)=\mu(qQ)-b(Q)\epsilon_{\mu}(Q)$.}{\LARGE \par}

{\LARGE{}On the other hand,}{\LARGE \par}

{\LARGE{}$\begin{array}{ccc}
\mu(\partial_{b(Q)}(qQ)) & = & \mu(\partial_{b(Q)}(P_{1}P_{2})-\partial_{b(Q)}(r))\\
 & \geq & \min\left\{ \mu(\partial_{b(Q)}(P_{1}P_{2})),\mu(\partial_{b(Q)}(r))\right\} \\
 & \geq & \min\left\{ \mu\left(\sum\limits _{j=0}^{b(Q)}\partial_{j}P_{1}\partial_{b(Q)-j}P_{2}\right),\mu(\partial_{b(Q)}r)\right\} \\
 & > & \min\limits _{0\leq j\leq b(Q)}\{\mu(P_{1})-j\epsilon_{\mu}(Q)+\mu(P_{2})-(b(Q)-j)\epsilon_{\mu}(Q),\mu(r)-b(Q)\epsilon_{\mu}(Q)\}\\
 & \ge & \mu(qQ)-b(Q)\epsilon_{\mu}(Q),
\end{array}$}{\LARGE \par}

\noindent {\LARGE{}which gives the desired contradiction. We have
proved that $\mu(P_{1}P_{2})=\mu(r)<\mu(qQ)$, so the Proposition holds
in the case $t=2$.}{\LARGE \par}

{\LARGE{}Assume, inductively, that $t>2$ and that the Proposition
is true for $t-1$. Let $P:=\prod\limits _{i=1}^{t-1}P_{i}$. Let
$P=q_{1}Q+r_{1}$ and $r_{1}P_{t}=q_{2}Q+r$ be the Euclidean divisions
by $Q$ of $P$ and $r_{1}P_{t}$, respectivly. Note that $q=q_{1}P_{t}+q_{2}$.
By the induction assumption, we have $\mu(r_{1})=\mu(P)<\mu(q_{1}Q)$,
hence $\mu(r_{1}P_{t})=\mu\left(\prod\limits _{i=1}^{t}P_{i}\right)<\mu(q_{1}P_{t}Q)$.
By the case $t=2$ we have $\mu(r_{1}P_{t})=\mu(r)<\mu(q_{2}Q)$.}{\LARGE \par}

{\LARGE{}Hence $\mu(r)=\mu(r_{1}P_{t})=\mu\left(\prod\limits _{i=1}^{t}P_{i}\right)<\min\left\{ \mu(q_{1}P_{t}Q),\mu(q_{2}Q)\right\} \leq\mu(q_{1}P_{t}Q+q_{2}Q)=\mu(qQ)$. }{\LARGE \par}\end{proof}
\begin{defn}
{\LARGE{}\label{defS}Let $Q$ be an abstract key polynomial for $\mu$,
$g$ an element of $K[x]$ and $g=\sum\limits _{j=0}^{s}g_{j}Q^{j}$
the $Q$-expansion of $g$. Put $S_{Q}(g)=\left\{ j\in\{0,\ldots,s\}\text{ such that }\mu(g_{j}Q^{j})=\mu{}_{Q}(g)\right\} $
and $\delta_{Q}(g):=\max S_{Q}(g)$. }{\LARGE \par}\end{defn}
\begin{prop}
{\LARGE{}If $Q$ is an abstract key polynomial, then $\mu_{Q}$ is
a valuation. }{\LARGE \par}\end{prop}
\begin{proof}
{\LARGE{}First, for any polynomials $f$ and $g$, we have 
\begin{equation}
\mu_{Q}(f+g)\geq\min\left\{ \mu_{Q}(f),\mu_{Q}(g)\right\} .\label{eq:largeinequality}
\end{equation}
We want to show that 
\begin{equation}
\mu_{Q}(fg)=\mu_{Q}(f)+\mu_{Q}(g).\label{eq:additivity}
\end{equation}
If both $f$ and $g$ have degree strictly less than $\deg(Q)$, we
have $\mu_{Q}(f)=\mu(f)$ and $\mu_{Q}(g)=\mu(g)$. Furthermore, by
Proposition \ref{rem:prod}, $\mu_{Q}(fg)=\mu(fg)$. Since $\mu$
is a valuation, (\ref{eq:additivity}) holds.}{\LARGE \par}

{\LARGE{}Next, let $i$, and $j$ be two non-negative integers. Let
$f_{i}$ and $g_{j}$ be two polynomials of degree strictly less than
$\deg(Q)$ and let $f_{i}g_{j}=aQ+b$ be the Euclidean division of
$f_{i}g_{j}$ be $Q$. We have $\deg\ a,\deg\ b<\deg\ Q$ and 
\begin{equation}
\mu(f_{i}g_{j})=\mu(b)<\mu(aQ)\label{eq:prop10}
\end{equation}
(by Proposition \ref{rem:prod}). Then $(f_{i}Q^{i})(g_{j}Q^{j})=aQ^{i+j+1}+bQ^{i+j}$
is a $Q$-expansion of $(f_{i}Q^{i})(g_{j}Q^{j})$. By definition
of $\mu_{Q}$ and (\ref{eq:prop10}) we have 
\begin{equation}
\mu_{Q}(f_{i}g_{j}Q^{i+j})=\mu\left(bQ^{i+j}\right)=\mu(f_{i}Q^{i})+\mu(g_{j}Q^{j})=\mu_{Q}(f_{i}Q^{i})+\mu_{Q}(g_{j}Q^{j}),\label{eq:monomialcase}
\end{equation}
which proves the equality (\ref{eq:additivity}) for $f=f_{i}Q^{i}$
and $g=g_{j}Q^{j}$ with $\deg\ f_{i},\deg\ g_{j}<\deg\ Q$.}{\LARGE \par}

{\LARGE{}It remains to show the equality (\ref{eq:additivity}) for
arbitrary polynomials $f=\sum\limits _{j=0}^{n}f_{j}Q^{j}$ and $g=\sum\limits _{j=0}^{m}g_{j}Q^{j}$.
It is sufficient to consider the case when all the terms in the $Q$-expansion
of $f$ have the same value and similarly for $g$. In other words,
we may replace $f$ and $g$ by $\mathrm{In_{Q}}f$ and $\mathrm{In_{Q}}g$,
respectively. By (\ref{eq:largeinequality}), (\ref{eq:monomialcase})
and the distributive law, we have 
\begin{equation}
\mu_{Q}(fg){\geq}\mu_{Q}(f)+\mu_{Q}(g).\label{eq:muQinequality}
\end{equation}
It remains to show that (\ref{eq:muQinequality}) is, in fact, an
equality. Let $n_{0}:=\min S_{Q}(f)$ and $m_{0}:=\min S_{Q}(g)$.
We denote by 
\[
f_{n_{0}}g_{m_{0}}=qQ+r
\]
the $Q$-expansion of $f_{n_{0}}g_{m_{0}}$. Hence the $Q$-expansion
of $\mathrm{In}_{Q}(f)\mathrm{In}_{Q}(g)$ contains the term $rQ^{n_{0}+m_{0}}$,
which, by Proposition \ref{rem:prod}, is of value 
$$\mu_{Q}(rQ^{n_{0}+m_{0}})=\mu(rQ^{n_{0}+m_{0}})=\mu(f_{n_{0}}Q^{n_{0}}g_{m_{0}}Q^{m_{0}})=\mu_{Q}(f)+\mu_{Q}(g)$$.
This completes the proof. }{\LARGE \par}\end{proof}
\begin{rem}
{\LARGE{}\label{transc} Let $\alpha:=\deg_{x}Q$. We define $G_{<\alpha}:=\sum\limits _{\deg_{x}P<\alpha}(\mathrm{in}_{\mu_{Q}}P)\mathrm{gr}_{\mu}K\subset\mathrm{gr}_{\mu_{Q}}K[x]$.
It follows from the $t=2$ case of Proposition \ref{rem:prod} that
$G_{<\alpha}$ is closed under multiplication, so it is, in fact,
a ring. The ring $G_{<\alpha}$ embeds into $\mathrm{gr}_{\mu}K[x]$
by the natural map which sends $\mathrm{in}_{\mu_{Q}}f$ to $\mathrm{in}_{\mu}f$
for each polynomial $f$ of degree strictly less than $\alpha$. We
have 
\[
\mathrm{gr}_{\mu_{Q}}K[x]=G_{<\alpha}[\mathrm{in}_{\mu_{Q}}Q],
\]
where $\mathrm{in}_{\mu_{Q}}Q$ is transcendental over $G_{<\alpha}$.
In particular, $\mathrm{in}_{\mu_{Q}}Q$ is irreducible in $\mathrm{gr}_{\mu_{Q}}K[x]$. }{\LARGE \par}\end{rem}
\begin{lem}
{\LARGE{}\label{inegalite} For every polynomial $f\in K[x]$ and
every $b\in\mathbb{N}^{\ast}$ we have 
$$\mu_{Q}(\partial_{b}f)\geq\mu_{Q}(f)-b\epsilon_{\mu}(Q)$$. }{\LARGE \par}\end{lem}
\begin{proof}
{\LARGE{}Let $f=\sum\limits _{j=0}^{s}f_{j}Q^{j}$ be the $Q$-expansion
of $f$.}{\LARGE \par}

{\LARGE{}It is enough to show the result for $f=f_{j}Q^{j}$. Indeed,
if we have the result in this case, then}{\LARGE \par}

{\LARGE{}$\begin{array}{ccc}
\mu_{Q}(\partial_{b}f) & = & \mu_{Q}\left(\sum\limits _{j=0}^{s}\partial_{b}(f_{j}Q^{j})\right)\\
 & \geq & \min\limits _{0\leq j\leq s}\left\{ \mu_{Q}\left(\partial_{b}(f_{j}Q^{j})\right)\right\} \\
 & \geq & \min\limits _{0\leq j\leq s}\left\{ \mu_{Q}(f_{j}Q^{j})-b\epsilon_{\mu}(Q)\right\} \\
 & \geq & \min\limits _{0\leq j\leq s}\left\{ \mu_{Q}(f_{j}Q^{j})\right\} -b\epsilon(Q)\\
 & \geq & \mu(f)-b\epsilon_{\mu}(Q).
\end{array}$}{\LARGE \par}

{\LARGE{}Now, let us show the result for $f=f_{j}Q^{j}$. First, we
show it for $f=f_{j}$.
\\Indeed, $\epsilon_{\mu_{Q}}(f_{j})=\epsilon_{\mu}(f_{j})<\epsilon_{\mu}(Q)$
since $Q$ is an abstract key polynomial of degree strictly superior
than $\deg(f_{j})$. Hence 
\[
\mu_{Q}(\partial_{b}f_{j})=\mu(\partial_{b}f_{j})>\mu_{Q}(f_{j})-b\epsilon_{\mu}(Q).
\]
This proves the Lemma with $f$ replaced by $f_{j}$.}{\LARGE \par}

{\LARGE{}We have $\mu_{Q}(\partial_{b}Q)\geq\mu_{Q}(Q)-b\epsilon_{\mu}(Q)$.}{\LARGE \par}

{\LARGE{}To finish the proof of the Lemma, it remains to show that
if we have the result for two polynomials $f$ and $g$, we have the
result for the product $fg$. Let us suppose that we have the result
for two polynomials $f$ and $g$.}{\LARGE \par}

{\LARGE{}Then,}{\LARGE \par}

$\begin{array}{ccc}
\mu_{Q}(\partial_{b}(fg)) & = & \mu_{Q}\left(\sum\limits _{s=0}^{b}\partial_{s}f\partial_{b-s}g\right)\\
 & \geq & \min\limits _{0\leq s\leq b}\left\{ \mu_{Q}(\partial_{s}f)+\mu_{Q}(\partial_{b-s}(g))\right\} \\
 & \geq & \min\limits _{0\leq s\leq b}\left\{ \mu_{Q}(\partial_{s}f)\right\} +\min\limits _{0\leq s\leq b}\left\{ \mu_{Q}(\partial_{b-s}(g))\right\} \\
 & \geq & \mu_{Q}(f)-s\epsilon_{\mu}(Q)+\mu_{Q}(g)-(b-s)\epsilon_{\mu}(Q)\\
 & \geq & \mu_{Q}(fg)-b\epsilon_{\mu}(Q)
\end{array}$.

{\LARGE{}This completes the proof. }{\LARGE \par}\end{proof}
\begin{prop}
{\LARGE{}\label{THEprop} Let the notation be as in Definition \ref{defS}.
If $S_{Q}(g)\neq\{0\}$ then there exists $b\in\mathbb{N}^{\ast}$
such that $\frac{\mu_{Q}(g)-\mu_{Q}(\partial_{b}g)}{b}=\epsilon_{\mu}(Q)$. }{\LARGE \par}\end{prop}
\begin{proof}
{\LARGE{}First, replacing $g$ by $\mathrm{In}_{Q}(g)=\sum\limits _{j\in S_{Q}(g)}g_{j}Q^{j}$
does not change the problem. We want to show the existence of a strictly
positive integer $b$ such that $\mu_{Q}(\partial_{b}g)=\mu_{Q}(g)-b\epsilon_{\mu}(Q)$.}{\LARGE \par}

{\LARGE{}Let $l$ be the minimum of $S_{Q}(g)\setminus\{0\}$. Write
$l=p^{e}u$, with $p\nmid u$. Let $b:=p^{e}b(Q)\in\mathbb{N}^{\ast}$.
We calculate $\partial_{b}g$. }{\LARGE \par}
\begin{lem}
{\LARGE{}We have $\partial_{b}g=urQ^{l-p^{e}}+Q^{l-p^{e}+1}R+S$,
where: }{\LARGE \par}\end{lem}
\begin{enumerate}
\item {\LARGE{}$r$ is the remainder of the Euclidean division of $g_{l}\left(\partial{}_{b(Q)}Q\right)^{p^{e}}$
by $Q$. }{\LARGE \par}
\item {\LARGE{}$R\in K[x]$. }{\LARGE \par}
\item {\LARGE{}$S\in K[x]$ and $\mu_{Q}(S)>\mu_{Q}(g)-b\epsilon_{\mu}(Q)$. }{\LARGE \par}\end{enumerate}
\begin{proof}
{\LARGE{}First, let us show that the Lemma holds for $g=g_{l}Q^{l}$
and that for every integer $j\in S_{Q}(g)\setminus\{l\}$, we have
$\partial_{b}(g_{j}Q^{j})=Q^{l-p^{e}+1}R_{j}+S_{j}$, where $\left(R_{j},S_{j}\right)\in K[x]^{2}$
such that $\mu_{Q}(S_{j})>\mu_{Q}(g)-b\epsilon_{\mu}(Q)$.}{\LARGE \par}

{\LARGE{}Let us prove these two statements together.}{\LARGE \par}

{\LARGE{}First, put $M_{j}:=\left\{ \left(b_{0},\ldots,b_{s}\right)\in\mathbb{N}^{s+1}\text{, }b_{0}+\ldots+b_{s}=b\text{, }s\leq j\right\} $.}{\LARGE \par}

{\LARGE{}By the Leibnitz rule,}{\LARGE \par}

{\LARGE{}$\partial_{b}(g_{j}Q^{j})=\sum\limits _{\begin{array}{c}
\left(b_{0},\ldots,b_{s}\right)\in M_{j}\end{array}}\underset{:=T(b_{0},\ldots,b_{s})}{\left(\underbrace{C(b_{0},\ldots,b_{s})\partial_{b_{0}}g_{j}\left(\prod\limits _{i=1}^{s}\partial_{b_{i}}Q\right)Q^{j-s}}\right)}$. where $C(b_{0},\ldots,b_{s})$ are certain integers whose exact
values can be found in \cite{HMOS}. Here by ``integer'' we mean
an element of the image of the natural map $\mathbb{N}\rightarrow K$,
that is, an element of $\mathbb{N}$ or $\mathbb{F}_{p}$ depending
on whether the characteristic of $K$ is 0 or $p>0$.}{\LARGE \par}

{\LARGE{}Put $N_{j}:=\left\{ \left(b_{0},\ldots,b_{s}\right)\in M_{j}\text{ such that }b_{0}>0\text{ or }\left\{ b_{1},\ldots,b_{s}\right\} \nsubseteq I(Q)\right\} $,}{\LARGE \par}

{\LARGE{}$S_{j}:=\sum\limits _{\left(b_{0},\ldots,b_{s}\right)\in N_{j}}T(b_{0},\ldots,b_{s})$,
$\alpha:=(0,\underset{p^{e}}{\underbrace{b(Q),\ldots,b(Q)}})$ and}{\LARGE \par}

{\LARGE{}$Q^{l-p^{e}+1}R_{j}:=\begin{cases}
\sum\limits _{\left(b_{0},\ldots,b_{s}\right)\in M_{j}\setminus N_{j}}T(b_{0},\ldots,b_{s}) & \text{if }j\neq l\\
\sum\limits _{\begin{array}{c}
\left(b_{0},\ldots,b_{s}\right)\in M_{j}\setminus N_{j}\\
(b_{0},\ldots,b_{s})\neq\alpha
\end{array}}T(b_{0},\ldots,b_{s}) & \text{if }j=l
\end{cases}$}{\LARGE \par}

{\LARGE{}If $j=l$, the number of times the term $T(0,\underset{p^{e}}{\underbrace{b(Q),\ldots,b(Q)}})$
appears in $\partial_{b}(g_{l}Q^{l})$ is $\binom{l}{p^{e}}=u$. Performing
the Euclidean division of $T(0,\underset{p^{e}}{\underbrace{b(Q),\ldots,b(Q)}})$
by $Q$, we obtain 
$$T(0,\underset{p^{e}}{\underbrace{b(Q),\ldots,b(Q)}})=R_{0}Q^{l-p^{e}+1}+urQ^{l-p^{e}}$$.}{\LARGE \par}

{\LARGE{}We are now in the position to calculate $\partial_{b}g$:}{\LARGE \par}

{\LARGE{}$\begin{array}{ccc}
\partial_{b}g & = & \partial_{b}\left(\sum\limits _{j\in S_{Q}(g)}g_{j}Q^{j}\right)\\
 & = & \partial_{b}(g_{l}Q^{l})+\sum\limits _{j\in S_{Q}(g)\setminus\{l\}}\partial_{b}(g_{j}Q^{j})\\
 & = & urQ^{l-p^{e}}+Q^{l-p^{e}+1}R_{l}+S_{l}+\sum\limits _{j\in S_{Q}(g)\setminus\{l\}}\left(Q^{l-p^{e}+1}R_{j}+S_{j}\right)\\
 & = & urQ^{l-p^{e}}+Q^{l-p^{e}+1}\left(\underset{:=R}{\underbrace{R_{l}+\sum\limits _{j\in S_{Q}(g)\setminus\{l\}}R_{j}}}\right)+\underset{:=S}{\underbrace{S_{l}+\sum\limits _{j\in S_{Q}(g)\setminus\{l\}}S_{j}}}
\end{array}$}{\LARGE \par}

{\LARGE{}with}{\LARGE \par}

{\LARGE{}$\begin{array}{ccc}
\mu_{Q}(S) & \geq & \min\left\{ \mu_{Q}(S_{l}),\mu_{Q}\left(\sum\limits _{j\in S_{Q}(g)\setminus\{l\}}S_{j}\right)\right\} \\
 & \geq & \min\limits _{j\in S_{Q}(g)}\left\{ \mu_{Q}(S_{j}),\right\} \\
 & > & \mu_{Q}(g)-b\epsilon_{\mu}(Q).
\end{array}$}{\LARGE \par}

{\LARGE{}This completes the proof of the Lemma. }{\LARGE \par}
\end{proof}
{\LARGE{}Next, in view of Lemma \ref{inegalite}, we have $\mu_{Q}(\partial_{b}g)\geq\mu_{Q}(g)-b\epsilon_{\mu}(Q)$.}{\LARGE \par}

{\LARGE{}Hence the $Q$-expansion of $\partial_{b}g$ contains the
term $urQ^{l-p^{e}}$ and terms wich either are divisible by $Q^{l-p^{e}+1}$
or have value greater than $\mu_{Q}(g)-b\epsilon_{\mu}(Q)$. To complete
the proof of the Proposition, it is sufficient to show that $\mu_{Q}(urQ^{l-p^{e}})=\mu_{Q}(rQ^{l-p^{e}})=\mu_{Q}(g)-b\epsilon_{\mu}(Q)$.}{\LARGE \par}

{\LARGE{}By Proposition \ref{rem:prod} we have $\mu(r)=\mu_{Q}(r)=\mu(g_{l}\left(\partial_{b(Q)}Q)^{p^{e}}\right)$,
hence}{\LARGE \par}

$\begin{array}{ccccc}
\mu_{Q}\left(rQ^{l-p^{e}}\right) & = & \mu\left(rQ^{l-p^{e}}\right) & = & \mu\left(g_{l}\left(\partial_{b(Q)}Q\right)^{p^{e}}Q^{l-p^{e}}\right) \\
& = & \mu(g_{l}Q^{l})+p^{e}\mu\left(\partial_{b(Q)}Q\right)-p^{e}\mu(Q)& = &\mu(g_{l}Q^{l})-p^{e}b(Q)\epsilon_{\mu}(Q) \\
 & = &  \mu_{Q}(g)-b\epsilon_{\mu}(Q).& &
\end{array}$

{\LARGE{}This completes the proof. }{\LARGE \par}\end{proof}
\begin{rem}
{\LARGE{}It can be shown that the implication of Proposition \ref{THEprop}
is, in fact, an equivalence. This will be accomplished in a forthcoming
paper. }{\LARGE \par}\end{rem}
\begin{cor}
{\LARGE{}\label{corcadeftt} Let $Q$ be an abstract key polynomial
and $f\in K[x]$. Suppose that there exists an integer $b\in\mathbb{N}^{\ast}$
such that $\frac{\mu_{Q}(f)-\mu_{Q}(\partial_{b}f)}{b}=\epsilon_{\mu}(Q)$
and $\mu_{Q}(\partial_{b}f)=\mu(\partial_{b}f)$. Then $\epsilon_{\mu}(f)\geq\epsilon_{\mu}(Q)$.}{\LARGE \par}

{\LARGE{}If moreover we have $\mu(f)>\mu_{Q}(f),$ then $\epsilon_{\mu}(f)>\epsilon_{\mu}(Q)$. }{\LARGE \par}\end{cor}
\begin{proof}
{\LARGE{}We have $\epsilon_{\mu}(f)\geq\frac{\mu(f)-\mu(\partial_{b}f)}{b}=\frac{\mu(f)-\mu_{Q}(\partial_{b}f)}{b}=\frac{\mu(f)+b\epsilon_{\mu}(Q)-\mu_{Q}(f)}{b}$. This means that $\epsilon_{\mu}(f)=\epsilon_{\mu}(Q)+\frac{\mu(f)-\mu_{Q}(f)}{b}\geq\epsilon_{\mu}(Q)$.
And if $\mu(f)>\mu_{Q}(f),$ then $\epsilon_{\mu}(f)>\epsilon_{\mu}(Q)$. }{\LARGE \par}\end{proof}
\begin{prop}
{\LARGE{}\label{prop:-irreducible-in} The polynomial $Q$ is $\mu_{Q}$-irreducible. }{\LARGE \par}\end{prop}
\begin{proof}
{\LARGE{}Put $L:=\mathrm{Frac}(G_{<\alpha})$ where $\alpha=\deg(Q)$.
Assume that 
\[
(\mathrm{in}_{\mu_{Q}}Q)(\mathrm{in}_{\mu_{Q}}c)=(\mathrm{in}_{\mu_{Q}}g)(\mathrm{in}_{\mu_{Q}}h)\in L[\mathrm{in}_{\mu_{Q}}Q].
\]
Then there exists $\lambda\in L^{*}$, such that $\mathrm{in}_{\mu_{Q}}Q=\lambda\mathrm{in}_{\mu_{Q}}g$
or $\mathrm{in}_{\mu_{Q}}Q=\lambda\mathrm{in}_{\mu_{Q}}h$. Since
all of $\mathrm{in}_{\mu_{Q}}Q$, $\mathrm{in}_{\mu_{Q}}c$, $\mathrm{in}_{\mu_{Q}}g$,
$\mathrm{in}_{\mu_{Q}}h$ are homogeneuos elements of $G_{<\alpha}[\mathrm{in}_{\mu_{Q}}Q]$,
so is $\lambda$. This proves that $Q$ is $\mu_{Q}$-irreducible. }{\LARGE \par}\end{proof}
\begin{prop}
{\LARGE{}\label{justepour102}Let $Q$ and $Q'$ be abstract key polynomials
such that $\epsilon_{\mu}(Q)\leq\epsilon_{\mu}(Q')$ and let $f\in K[x]$.}{\LARGE \par}

{\LARGE{}Then $\mu_{Q}(f)\leq\mu_{Q'}(f)$. If $\mu_{Q}(f)=\mu(f)$,
then $\mu_{Q'}(f)=\mu(f)$. }{\LARGE \par}\end{prop}
\begin{proof}
{\LARGE{}First, we show that $\mu_{Q'}(Q)=\mu(Q)$. If $\deg_{x}(Q)<\deg_{x}(Q')$,
this is clear. Otherwise, we have $\deg_{x}(Q)=\deg_{x}(Q')$, since
$Q$ is an abstract key polynomial and $\epsilon_{\mu}(Q)\leq\epsilon_{\mu}(Q')$.
Let us suppose that $\mu_{Q'}(Q)<\mu(Q)$. Then $S_{Q'}(Q)\neq\left\{ 0\right\} $.
In view of Proposition \ref{THEprop} and Corollary \ref{corcadeftt},
we have $\epsilon_{\mu}(Q)>\epsilon_{\mu}(Q')$, which is a contradiction.}{\LARGE \par}

{\LARGE{}Now let $f=\sum\limits _{j=0}^{s}f_{j}Q^{j}$ be the $Q$-expansion
of $f$. For each integer $j\in\left\{ 0,\ldots,s\right\} $, we have
\[
\mu_{Q'}(f_{j}Q^{j})=\mu_{Q'}(f_{j})+j\mu_{Q'}(Q)=\mu_{Q'}(f_{j})+j\mu(Q).
\]
Then, since $\deg_{x}(f_{j})<\deg_{x}(Q)\leq\deg_{x}(Q')$, we have
$\mu_{Q'}(f_{j}Q^{j})=\mu(f_{j})+j\mu(Q)=\mu(f_{j}Q^{j})$. Hence
$\mu_{Q'}(f)\geq\min\limits _{0\leq j\leq s}\left\{ \mu_{Q'}(f_{j}Q^{j})\right\} =\min\limits _{0\leq j\leq s}\left\{ \mu(f_{j}Q^{j})\right\} =\mu_{Q}(f)$.}{\LARGE \par}

{\LARGE{}Let us now suppose that $\mu_{Q}(f)=\mu(f)\leq\mu_{Q'}(f)$.
As we know that $\mu_{Q'}(f)\leq\mu(f)$, we obtain 
\[
\mu_{Q}(f)=\mu_{Q'}(f),
\]
as desired. }{\LARGE \par}\end{proof}
\begin{prop}
{\LARGE{}\label{Proposition10.2} Let $f_{1},\ldots,f_{r}\in K[x]$
be polynomials and let $n:=\max\limits _{1\leq i\leq r}\left\{ \deg_{x}(f_{i})\right\} $.}{\LARGE \par}

{\LARGE{}Then there exists an abstract key polynomial $Q$ of degree
less than or equal to $n$ such that for each integer $i\in\left\{ 1,\ldots,r\right\} $,
we have $\mu_{Q}(f_{i})=\mu(f_{i})$. }{\LARGE \par}\end{prop}
\begin{proof}
{\LARGE{}First, we show that it is sufficient to prove the Proposition
for $r=1$.}{\LARGE \par}

{\LARGE{}Indeed, suppose the Proposition proved when there is just
one polynomial and suppose $r>1$. Hence we can find $Q_{1},\ldots,Q_{r}$
abstract key polynomials of degrees less or equal than $n$ such that
for each integer $i\in\left\{ 1,\ldots,r\right\} $, we have $\mu_{Q_{i}}(f_{i})=\mu(f_{i})$.}{\LARGE \par}

{\LARGE{}Renumbering the $Q_{i}$, if necesssary, we may assume that
$\epsilon_{\mu}(Q_{r})\geq\epsilon_{\mu}(Q_{i})$ for every integer
$i\in\left\{ 1,\ldots,r\right\} $. By Proposition \ref{justepour102},
we have, for each $i\in\left\{ 1,\ldots,r\right\} $, $\mu_{Q_{r}}(f_{i})=\mu(f_{i})$.}{\LARGE \par}

{\LARGE{}Let us show the case $r=1$. We argue by contradiction. Suppose
that there exists a polynomial $f$ such that for every abstract key
polynomial $Q$ of degree less than or equal to $\deg_{x}(f)$, we
have $\mu_{Q}(f)<\mu(f)$. Choose $f$ of minimal degree among the
polynomials having this property. }{\LARGE \par}
\begin{claim*}
{\LARGE{}There exists an abstract key polynomial $Q$ of degree less
than or equal to $\deg_{x}f$ such that 
\[
\mu_{Q}(\partial_{b}f)=\mu(\partial_{b}f)
\]
for every $b\in\mathbb{N}^{\ast}$.}{\LARGE \par}

{\LARGE{}Indeed, let $s=\deg_{x}f$, so that for each integer $j$
strictly greater than $s$, we have $\partial_{j}f=0$. By the minimality
assumption on $\deg_{x}f$, for each $i\in\{1,\dots,s\}$ there exists
an abstract key polynomial $Q_{i}$ such that $\mu_{Q_{i}}(\partial_{j}f)=\mu(\partial_{j}f)$,.}{\LARGE \par}

{\LARGE{}Take an $i\in\{1,\dots,s\}$ such that $\epsilon_{\mu}(Q_{i})=\max\limits _{1\leq j\leq s}\left\{ \epsilon_{\mu}(Q_{j})\right\} $.
Then, in view of Proposition \ref{justepour102}, for each integer
$1\leq j\leq s$, we have $\mu_{Q_{i}}(\partial_{j}f)=\mu(\partial_{j}f)$,
and the Claim follows. }{\LARGE \par}
\end{claim*}
{\LARGE{}Now, we have $\mu_{Q}(f)<\mu(f)$, so in particular $S_{Q}(f)\neq\left\{ 0\right\} $,
and for each $b\in\mathbb{N}^{\ast}$, $\mu_{Q}(\partial_{b}f)=\mu(\partial_{b}f)$.
In view of Proposition \ref{THEprop} and Corollary \ref{corcadeftt},
we have 
\begin{equation}
\epsilon_{\mu}(f)>\epsilon_{\mu}(Q).\label{eq:strictepsilon}
\end{equation}
We claim that the last inequality is true for every abstract key polynomial
of degree less than or equal to $\deg_{x}f$.}{\LARGE \par}

{\LARGE{}Indeed, let us take $Q'$ an abstract key polynomial of degree
less than or equal to $\deg_{x}f$. We have two cases.}{\LARGE \par}

{\LARGE{}First case: $\epsilon_{\mu}(Q')\le\epsilon_{\mu}(Q)$. In
view of (\ref{eq:strictepsilon}), we have $\epsilon_{\mu}(Q')\le\epsilon_{\mu}(Q)<\epsilon_{\mu}(f)$.}{\LARGE \par}

{\LARGE{}Second case: $\epsilon_{\mu}(Q)<\epsilon_{\mu}(Q')$. In
view of Proposition \ref{justepour102}, we have $\mu(\partial_{b}f)=\mu_{Q}(\partial_{b}f)=\mu_{Q'}(\partial_{b}f)$
for each strictly positive integer $b$. Since $\mu_{Q'}(f)<\mu(f)$,
arguing as before, we have $\epsilon_{\mu}(Q')<\epsilon_{\mu}(f)$.}{\LARGE \par}

{\LARGE{}By definition of the abstract key polynomials, there exists
an abstract key polynomial $Q'$ of degree less than or equal to $\deg_{x}f$
such that $\epsilon_{\mu}(f)\le\epsilon_{\mu}(Q')$. This is a contradiction. }{\LARGE \par}
\end{proof}

\section{The relationship between the abstract and the Mac Lane\textendash Vaqui\'e
key polynomials.}

{\LARGE{}\label{relationship}}{\LARGE \par}

{\LARGE{}The aim of this section is to study the relationship between
the abstract and the Mac Lane\textendash Vaqui\'e key polynomials. }{\LARGE \par}
\begin{defn}
{\LARGE{}Let $Q$ and $Q'$ be two abstract key polynomials such that
$\epsilon_{\mu}(Q)<\epsilon_{\mu}(Q')$. We say that $Q'$ is an }\textbf{\LARGE{}immediate
successor}{\LARGE{} of $Q$ and we write $Q<Q'$ if $\deg_{x}(Q')$
is minimal among all the $Q'$ which satisfy $\epsilon_{\mu}(Q)<\epsilon_{\mu}(Q')$. }{\LARGE \par}\end{defn}
\begin{thm}
{\LARGE{}\label{abstractimpliesVaquie} Let $Q$ be an abstract key
polynomial for $\mu$. Then $Q$ is a Mac Lane \textendash{} Vaqui\'e
key polynomial for $\mu_{Q}$. }{\LARGE \par}\end{thm}
\begin{proof}
{\LARGE{}We have to prove two things:}{\LARGE \par}

{\LARGE{}1. $Q$ is $\mu_{Q}$-irreducible.}{\LARGE \par}

{\LARGE{}2. $Q$ is $\mu_{Q}$-minimal.}{\LARGE \par}

{\LARGE{}Statement 1 is nothing but Proposition \ref{prop:-irreducible-in}.}{\LARGE \par}

{\LARGE{}Now we are going to show the statement 2. We assume that
$Q\ |_{\mu_{Q}}r$, We want to show that $\deg_{x}r\geq\deg_{x}Q$.}{\LARGE \par}

{\LARGE{}By assumption, there exists $c$ such that 
$$\mathrm{(in_{\mu_{Q}}Q)(}\mathrm{in_{\mu_{Q}}c)=\mathrm{in}_{\mu_{Q}}r\in}\mathrm{gr}_{\mu_{Q}}K[x]\subset G_{<\alpha}[\mathrm{in}_{\mu_{Q}}Q]\subset L[\mathrm{in}_{\mu_{Q}}Q]$$.
Since $\mathrm{in}_{\mu_{Q}}Q$ is transcendental over $L$, we have
\begin{equation}
\deg_{\mathrm{in}_{\mu_{Q}}Q}\mathrm{(in}_{\mu_{Q}}r)\geq1.\label{eq:degree>1}
\end{equation}
Let $r=\sum\limits _{j=0}^{n}r_{j}Q^{j}$ be the $Q$-expansion of
$r$. By the algebraic independence of $\mathrm{in}_{\mu_{Q}}Q$ over
$L$ (and hence, }\textit{\LARGE{}a fortiori}{\LARGE{}, over $G_{<\alpha}$),
we have $\mathrm{in}_{\mu_{Q}}r=\sum\limits _{j=0}^{n}\mathrm{in}_{\mu_{Q}}r_{j}\mathrm{in}_{\mu_{Q}}Q^{j}$.
combined with (\ref{eq:degree>1}), this shows that $n\ge1$. We obtain
\[
\deg_{x}r=n\deg_{x}Q+\deg_{x}r_{n}\geq\deg_{x}Q+\deg_{x}r_{n}\geq\deg_{x}Q.
\]
This completes the proof. }{\LARGE \par}\end{proof}
\begin{lem}
{\LARGE{}\label{epsilonimpliesmu} Let $Q$ and $Q'$ be two abstract
key polynomials for $\mu$ such that $\epsilon_{\mu}(Q)<\epsilon_{\mu}(Q')$.
Then 
\[
\mu_{Q}(Q')<\mu(Q').
\]
}{\LARGE \par}\end{lem}
\begin{proof}
{\LARGE{}In view of Lemma \ref{inegalite}, we have $\frac{\mu_{Q}(Q')-\mu_{Q}(\partial_{b}Q')}{b}\leq\epsilon_{\mu}(Q)$
for each strictly positive integer $b$. Assume that $\mu_{Q}(Q')=\mu(Q')$,
aiming for contradiction. Then 
\[
\frac{\mu(Q')-\mu_{Q}(\partial_{b}Q')}{b}\leq\epsilon_{\mu}(Q),
\]
hence $\frac{\mu(Q')-\mu(\partial_{b}Q')}{b}\leq\epsilon_{\mu}(Q)$.
In other words, $\epsilon_{\mu}(Q')\leq\epsilon_{\mu}(Q)$, which
gives the desired contradiction. }{\LARGE \par}\end{proof}
\begin{prop}
{\LARGE{}\label{succimmnoyau} Let $Q$ and $Q'$ be two abstract
key polynomials for $\mu$. The following conditions are equivalent:}{\LARGE \par}

{\LARGE{}(1) $Q<Q'$}{\LARGE \par}

{\LARGE{}(2) $\mu_{Q}(Q')<\mu(Q')$ and $Q'$ is of minimal degree
with respect to this property. }{\LARGE \par}\end{prop}
\begin{proof}
{\LARGE{}(2)$\Longrightarrow$(1). Let us assume that $\mu_{Q}(Q')<\mu(Q')$
and that $Q'$ is of minimal degree minimal for this property. Then
$S_{Q}(Q')\neq\left\{ 0\right\} $ and for each strictly positive
integer $b$, we have $\mu_{Q}(\partial_{b}Q')=\mu(\partial_{b}Q')$.
By Proposition \ref{THEprop}, there exists $b\in\mathbb{N}^{\ast}$
such that $\frac{\mu_{Q}(Q')-\mu_{Q}(\partial_{b}g)}{b}=\epsilon_{\mu}(Q)$.
Hence $\frac{\mu_{Q}(Q')-\mu(\partial_{b}Q')}{b}=\epsilon_{\mu}(Q)$,
and $\epsilon_{\mu}(Q)<\frac{\mu(Q')-\mu(\partial_{b}Q')}{b}\leq\epsilon_{\mu}(Q')$.
If there exists a key polynomial $Q''$ satisfying $\epsilon_{\mu}(Q)<\epsilon_{\mu}(Q'')$
of degree strictly smaller than $\deg_{x}Q'$, by Lemma \ref{epsilonimpliesmu}
we would have $\mu_{Q}(Q'')<\mu(Q'')$, which would contradict the
minimality assumption on the degree of $Q'$. This proves (1).}{\LARGE \par}

{\LARGE{}(1)$\Longrightarrow$(2). Let us assume that $Q<Q'$. By
Lemma \ref{epsilonimpliesmu}, this implies that $\mu_{Q}(Q')<\mu(Q')$.
Moreover, if there existed an abstract key polynomial $Q''$ satisfying
$\mu_{Q}(Q'')<\mu(Q'')$ of degree strictly smaller than $\deg_{x}Q'$,
take such a $Q''$ of minimal degree. By the implication (2)$\Longrightarrow$(1)
of the Proposition we would have $\epsilon_{\mu}(Q)<\epsilon_{\mu}(Q'')$,
which would contradict the minimality assumption on the degree of
$Q'$. This proves (2). }{\LARGE \par}\end{proof}
\begin{thm}
{\LARGE{}\label{successorimpliesVaquie} Let $Q$ and $Q'$ be two
abstract key polynomials for $\mu$ such that $Q<Q'$. Then $Q'$
is a Mac Lane \textendash{} Vaqui\'e key polynomial for $\mu_{Q}$. }{\LARGE \par}\end{thm}
\begin{proof}
{\LARGE{}We have to prove two things:}{\LARGE \par}

{\LARGE{}1. $Q'$ is $\mu_{Q}$-irreducible.}{\LARGE \par}

{\LARGE{}2. $Q'$ is $\mu_{Q}$-minimal.}{\LARGE \par}

{\LARGE{}First we show 1. Let $\alpha=\deg_{x}Q$. By Remarks \ref{UFD}
and \ref{transc}, it is sufficient to show that $\mathrm{in}_{\mu_{Q}}(Q')$
is irreducible in 
\[
G_{<\alpha}[\mathrm{in}_{\mu_{Q}}(Q)]=\mathrm{gr}_{\mu_{Q}}K[x].
\]
Let $\varphi\colon\mathrm{gr}_{\mu_{Q}}K[x]\to\mathrm{gr}_{\mu}K[x]$
be the natural map which sends $\mathrm{in}_{\mu_{Q}}(f)$ to $\mathrm{in}_{\mu}(f)$
for every polynomial $f$. The map $\varphi$ maps $G_{<\alpha}$
isomorphically onto its image in $\mathrm{gr}_{\mu}K[x]$. The map
$\varphi$ is not injective if and only if there exists a polynomial
$f$ such that $\mu_{Q}(f)<\mu(f)$. In view of Proposition \ref{succimmnoyau},
we have this property for $f=Q'$; in particular, $\mathrm{in}_{\mu_{Q}}(Q')\in\mathrm{Ker}(\varphi)$.
We claim that $\mathrm{Ker}(\varphi)$ is a principal prime ideal,
generated by $\mathrm{in}_{\mu_{Q}}(Q')$. Indeed, take any polynomial
$f$ such that $\mathrm{in}_{\mu_{Q}}f\in\mathrm{Ker}(\varphi)$ and
let $\mathrm{in}_{\mu_{Q}}f=\mathrm{in}_{\mu_{Q}}(a)\mathrm{in}_{\mu_{Q}}(Q')+\mathrm{in}_{\mu_{Q}}(r)$
be the Euclidean division of $\mathrm{in}_{\mu_{Q}}f$ by $\mathrm{in}_{\mu_{Q}}(Q')$.
Then, if $\mathrm{in}_{\mu_{Q}}(r)\neq0$ we have $\mathrm{in}_{\mu_{Q}}(r)\in\mathrm{Ker}(\varphi)$
and so $\mu_{Q}(r)<\mu(r)$, which contradicts the minimality of the
degree of $Q'$. Thus $\mathrm{Ker}(\varphi)=\left(\mathrm{in}_{\mu_{Q}}(Q')\right)\mathrm{gr}_{\mu_{Q}}K[x]$.
Since $\mathrm{gr}_{\mu}K[x]$ has no zero divisors, we know that $\mathrm{in}_{\mu_{Q}}(Q')$
is a prime ideal. Thus $\mathrm{in}_{\mu_{Q}}(Q')$ is irreducible
in $\mathrm{gr}_{\mu_{Q}}K[x]$. This completes the proof of 1.}{\LARGE \par}

{\LARGE{}Now we show 2. Assume that $Q'\ |_{\mu_{Q}}\,r$. We want
to show that $\deg_{x}(r)\geq\deg_{x}(Q')$. First, we know that $\mathrm{in}_{\mu_{Q}}(Q')$
divides $\mathrm{in}_{\mu_{Q}}(r)$ in the unique factorisation domain
$L[\mathrm{in}_{\mu_{Q}}Q]\supset\mathrm{G}_{<\alpha}[\mathrm{in}_{\mu_{Q}}Q]=\mathrm{gr}_{\mu_{Q}}K[x]$.
Hence $r\in\mathrm{Ker}(\varphi)$. In other words, $\mu_{Q}(r)<\mu(r)$.
On the other hand, we know that $\mu_{Q}(Q')<\mu(Q')$ and that $Q'$
is of minimal degree for this property in view of Proposition \ref{succimmnoyau}.
By the minimality of $\deg_{x}(Q')$, we get the result. }{\LARGE \par}\end{proof}
\begin{thm}
{\LARGE{}\label{vaquieimpliesabstract} Fix a monic polynomial $Q\in K[x]$.
Let $\mu'$ be a valuation of $K(x)$ such that:}{\LARGE \par}

{\LARGE{}1. For each $f$ of degree strictly less than $\deg(Q)$,
we have $\mu'(f)=\mu(f)$;}{\LARGE \par}

{\LARGE{}2. $\mu'(Q)>\mu(Q)$.}{\LARGE \par}

{\LARGE{}Then $Q$ is an abstract key polynomial for $\mu'$. }{\LARGE \par}\end{thm}
\begin{proof}
{\LARGE{}Assume that $Q$ is not an abstract key polynomial for $\mu'$.
Then there exists a monic polynomial $g$ such that 
\begin{equation}
\epsilon_{\mu'}(g)\geq\epsilon_{\mu'}(Q)\label{eq:nonstrictepsilon}
\end{equation}
and 
\begin{equation}
\deg(g)<\deg(Q).\label{eq:strictdegree}
\end{equation}
We can choose $g$ of minimal degree for this property, and hence
$g$ is an abstract key polynomial.}{\LARGE \par}

{\LARGE{}Thus there exists an abstract key polynomial $g$ such that
$\epsilon_{\mu'}(g)\geq\epsilon_{\mu'}(Q)$ and $\deg(g)<\deg(Q)$.}{\LARGE \par}

{\LARGE{}Since every derivative of $Q$ has degree strictly smaller
than $\deg(Q)$, we have 
\[
\deg_{x}g\leq\deg_{x}(\partial_{1}Q)=\max\limits _{b\in\mathbb{N}^{\ast}}\left\{ \deg_{x}(\partial_{b}Q)\right\} .
\]
By Proposition \ref{Proposition10.2}, replacing $g$ by another abstract
key polynomial with larger $\epsilon_{\mu}$, if necessary, we may
assume, in addition, that 
\begin{equation}
\mu'(\partial_{b}Q){=}\mu'_{g}(\partial_{b}Q)\label{eq:partialsofQ}
\end{equation}
for all strictly positive integers $b$ (at this point, the abstract
key polynomial $g$ still satisfies (\ref{eq:nonstrictepsilon}) and
(\ref{eq:strictdegree}) but we may no longer have the condition that
$g$ is of minimal degree for this property).}{\LARGE \par}

{\LARGE{}We claim that for each polynomial $h$, we have 
\begin{equation}
\mu(h){\geq}\mu'_{g}(h).\label{eq:muvsmu'_g}
\end{equation}
Indeed, let $h=\sum\limits _{j=0}^{l}h_{j}g^{j}$ be the $g$-expansion
of $h$. We have 
\[
\mu'_{g}(h)=\mu'_{g}\left(\sum\limits _{j=0}^{l}h_{j}g^{j}\right)=\min\limits _{0\leq j\leq l}\left\{ \mu'(h_{j})+j\mu'(g)\right\} =\min\limits _{0\leq j\leq l}\left\{ \mu(h_{j})+j\mu(g)\right\} 
\]
by hypothesis 1.}{\LARGE \par}

{\LARGE{}Hence $\mu'_{g}(h)=\min\limits _{0\leq j\leq l}\mu(h_{j}g^{j})\leq\mu(h)$.
In particular, $\mu'(Q)>\mu(Q)\geq\mu'_{g}(Q)$.}{\LARGE \par}

{\LARGE{}Recall that if $Q=\sum\limits _{j=0}^{s}Q_{j}g^{j}$ is the
$g$-expansion of $Q$, we denote 
$$S_{g}(Q)=\left\{ j\in\{0,\ldots,s\}\ \left|\ \mu'(Q_{j}g^{j})=\mu'_{g}(Q)\right.\right\}. $$
Suppose $S_{g}(Q)=\{0\}$, then $\mu'(Q)>\mu'_{g}(Q)=\mu'(Q_{0})$.
Hence 
\[
\mu'\left(Q_{0}\right)=\mu'\left(\sum_{j\geq1}Q_{j}g^{j}\right)\geq\min_{j\geq1}\mu'(Q_{j}g^{j})>\mu'(Q_{0}),
\]
which is a contradiction. We have proved that $S_{g}(Q)\neq\{0\}$.}{\LARGE \par}

{\LARGE{}By Proposition \ref{THEprop}, there exists a strictly positive
integer $b$ such that $\frac{\mu'_{g}(Q)-\mu'_{g}(\partial_{b}Q)}{b}=\epsilon_{\mu'}(g)$.
By virtue of (\ref{eq:partialsofQ}) we obtain $\epsilon_{\mu'}(g)=\frac{\mu'_{g}(Q)-\mu'_{g}(\partial_{b}Q)}{b}=\frac{\mu'_{g}(Q)-\mu'(\partial_{b}Q)}{b}$.
Since $\mu'(Q)>\mu'_{g}(Q)$, we have 
\[
\epsilon_{\mu'}(g)<\frac{\mu'(Q)-\mu'(\partial_{b}Q)}{b}=\epsilon_{\mu'}(Q),
\]
which is a contradiciton.}{\LARGE \par}

{\LARGE{}Hence $Q$ is an abstract key polynomial for $\mu'$. }{\LARGE \par}\end{proof}
\begin{prop}
{\LARGE{}\label{existeval} (\cite{V}, $\mathrm{Proposition\ 1.3}$)
Let $Q$ be a Mac Lane \textendash \ Vaqui\'e key polynomial for the
valuation $\mu$. Then there exists a valuation $\mu'$ such that: }{\LARGE \par}\end{prop}
\begin{enumerate}
\item {\LARGE{}For each $f$ of degree strictly less than $\deg(Q)$, we
have $\mu'(f)=\mu(f)$ }{\LARGE \par}
\item {\LARGE{}$\mu'(Q)>\mu(Q)$. }{\LARGE \par}\end{enumerate}
\begin{cor}
{\LARGE{}Let $Q$ be a Mac Lane \textendash \ Vaqui\'e key polynomial
for the valuation $\mu$. Then it is an abstract key polynomial for
any valuation $\mu'$ satisfying the conclusion of Proposition \ref{existeval}. }{\LARGE \par}
\end{cor}
{\LARGE{}\bibliographystyle{amsplain} }{\LARGE \par}

\begin{thebibliography}{1}
{\LARGE{}\bibitem{HOS} F. J. Herrera Govantes, M. A. Olalla Acosta,
M. Spivakovsky, {\em Valuations in algebraic field extensions},
Journal of Algebra, Volume 312, Issue 2 (2007), pages 1033-1074. }{\LARGE \par}

{\LARGE{}\bibitem{HMOS} F. J. Herrera Govantes, W. Mahboub, M. A.
Olalla Acosta, M. Spivakovsky }\emph{\LARGE{}Key polynomials for simple
extensions of valued fields.}{\LARGE{}, arXiv:1406.0657v3 {[}math.AG{]}. }{\LARGE \par}

{\LARGE{}\bibitem{ML1} S. Mac Lane, {\em A construction for prime
ideals as absolute values of an algebraic field}, Duke Math. J.,
vol 2 (1936), pages 492--510. }{\LARGE \par}

{\LARGE{}\bibitem{ML} S. MacLane {\em A construction for absolute
values in polynomial rings}, Transactions of the AMS, vol 40 (1936),
pages 363--395.}{\LARGE \par}


{\LARGE{}

{\LARGE{}\bibitem{Ma} W. Mahboub {\em Key Polynomials}, Journal of Pure and Applied Algebra, 217(6) (2013),
pages 989--1006.}{\LARGE \par}


{\LARGE{}\bibitem{Ma1} W. Mahboub {\em Une construction explicite de polyn\^omes-cl\'es pour des valuations de rang fini}, Th\`ese de Doctorat, Institut de Math\'ematiques de Toulouse, 2013.}{\LARGE \par}


{\LARGE{}\bibitem{JCSS} J.-C. San Saturnino, }\textit{\LARGE{}Defect
of an extension, key polynomials and local uniformization}{\LARGE{},
preprint, arXiv:1412.7697, 2014.}{\LARGE \par}


{\small \par}

{\LARGE{}\bibitem{V0} Vaqui\'e, Michel, {\em Famille admise associ\'ee
\`a une valuation de $K[x]$. {[}Admissible family associated with a
valuation of $K[x]${]}}, Singularit\'es Franco-Japonaises, 391--428,
S\'emin. Congr., 10, Soc. Math. France, Paris, 2005. }{\LARGE \par}

{\LARGE{}\bibitem{V} Vaqui\'e, Michel, {\em Extension d'une valuation.
{[}Extension of a valuation{]}}, Trans. Amer. Math. Soc. 359 (2007),
no. 7, 3439--3481. (electronic) }{\LARGE \par}

{\LARGE{}\bibitem{V1} Vaqui\'e, Michel, {\em Famille admissible de
valuations et d\'efaut d'une extension. {[}Admissible family of valuations
and defect of an extension{]}}, J. Algebra 311 (2007), no. 2, 859--876. }{\LARGE \par}

{\LARGE}\bibitem{V2} Vaqui\'e, Michel, {\em Extensions de valuation
et polygone de Newton. {[}Valuation extensions and Newton polygon{]}}
Ann. Inst. Fourier (Grenoble) 58 (2008), no. 7, 2503--2541.}{\LARGE \par}\end{thebibliography}

\end{document}